\documentclass[11pt, reqno]{amsart}

\usepackage{amsmath, amsthm, amssymb}
\usepackage{enumerate}
\usepackage[margin=1.5in]{geometry}
\usepackage{tikz}
\usepackage{tikz-cd}
\usepackage{ytableau}
\usepackage{mathrsfs}
\usepackage{arydshln}
\usepackage{graphicx}
\usepackage{mathdots}
\usepackage[colorlinks=true,urlcolor=blue,linkcolor=blue,citecolor=blue]{hyperref}
\usepackage{cleveref}

\title{Rational Points on Certain Homogeneous Varieties}
\author{Pengyu Yang}
\address{Department of Mathematics\\
	ETH Z\"urich\\
	Switzerland}
\email{pengyu.yang@math.ethz.ch}
\date{\today}

\newtheorem{thm}{Theorem}[section]
\newtheorem{prop}[thm]{Proposition}
\newtheorem{lem}[thm]{Lemma}
\newtheorem{cor}[thm]{Corollary}

\theoremstyle{definition}
\newtheorem{definition}[thm]{Definition}
\newtheorem{example}[thm]{Example}

\theoremstyle{remark}

\newcommand{\Q}{\mathbb{Q}}
\newcommand{\R}{\mathbb{R}}

\newcommand{\A}{\mathbb{A}}
\newcommand{\SL}{\mathrm{SL}}
\newcommand{\Sp}{\mathrm{Sp}}

\numberwithin{equation}{section}

\begin{document}
	\begin{abstract}
		Let $L$ be a simply-connected simple connected algebraic group over a number field $F$, and $H$ be a semisimple absolutely maximal connected $F$-subgroup of $L$. Let $\Delta(H)$ be the image of $H$ diagonally embedded in $L^n$. Under a cohomological condition, we prove an asymptotic formula for the number of rational points of bounded height on projective equivariant compactifications of $\Delta(H)\backslash L^n$ with respect to a balanced line bundle, and hence confirm Manin's conjecture for this case.
	\end{abstract}

	\maketitle

	\section{Introduction}
	Let $X$ be a smooth projective variety over a number field $F$. $X$ is called Fano if the anticanonical divisor $-K_X$ is ample. Given an ample line bundle $\mathscr{L}$ on $X$, and an adelic metrization $\mathcal{L}$ of $\mathscr{L}$, we can define an associated height function
	\begin{equation}
	H_{\mathcal{L}}\colon X(F)\to\R_{>0}
	\end{equation}
	on the set of $F$-rational functions (see \cite{CT} Section 2). Take a suitable Zariski open $X^{\circ}\subset X$, Manin's conjecture \cite{MT} predicts the asymptotic growth of the number of rational points with height at most $T$ in $X^{\circ}(F)$, as $T\to\infty$. Consider the counting function
	\begin{equation}
	N(X^{\circ}(F),\mathcal{L},T) = \#\{ x\in X^{\circ}(F)\colon H_{\mathcal{L}}(x)\leq T \},
	\end{equation}
	and it is conjectured that
	\begin{equation}
	N(X^{\circ}(F),\mathcal{L},T) \sim c(\mathcal{L})T^{a(\mathscr{L})}(\log T)^{b(\mathscr{L})-1},
	\end{equation}
	where $c(\mathcal{L})>0$, and $a(\mathscr{L})$, $b(\mathscr{L})$ are geometric constants attached to $X$ and $\mathscr{L}$, which we define below. \par

	Let $\mathscr{L}$ be an ample line bundle on $X$, and let $\Lambda_{\mathrm{eff}}(X)$ denote the pseudo-effective cone in the real N\'eron-Severi group $\mathrm{NS}(X,\R)$. We define
	\begin{align}\label{eq:def_a_b}
	a(\mathscr{L}) &= \inf \left\{ t\in\Q\colon t[\mathscr{L}]+[K_X]\in \Lambda_{\mathrm{eff}}(X) \right\},\\ \nonumber
	b(\mathscr{L}) &= \text{the maximal codimension of the face containing } a(\mathscr{L})[\mathscr{L}]+[K_X].
	\end{align}
	The property of these two constants was systematically studied in \cite{HTT}, where the notion of \emph{balanced} line bundle was defined.\par

	In the equivariant setting, let $G$ be an algebraic group over $F$, and $H$ be an $F$-subgroup of $G$. Take $X^{\circ}=H\backslash G$, and let $X$ be a smooth $G$-equivariant compactification of $X^{\circ}$. Let $\mathscr{L}$ be an ample line bundle on $X$. Several cases have been studied in recent years using ergodic-theoretical methods. Gorodnik, Maucourant and Oh \cite{GMO} proved Manin's conjecture for $G=H\times H$. Later Gorodnik and Oh \cite{GO} proved Manin's conjecture for $G$ a connected semisimple $F$-group, and $H$ a semisimple maximal connected $F$-subgroup of $G$, under certain cohomological condition. Gorodnik, Takloo-Bighash and Tschinkel \cite{GTT} settled the case where $H$ is a simple connected $F$-group diagonally embedded into $G=H^n$, and $\mathscr{L}=-K_X$ is the anticanonical bundle. \par 
	\begin{definition}
		Let $X$ be an equivariant compactification of $X^{\circ}=H\backslash G$ and $H'\subset G$ any closed proper subgroup containing the diagonal, i.e. $H\subset H'$. Let $X'\subsetneq X$ be the induced equivariant compactification of $H'$. A line bundle $\mathscr{L}$ on $X$ is called \emph{balanced with respect to $H'$} if
		\begin{equation}
		\left( a(\mathscr{L}|_{X'}),b(\mathscr{L}|_{X'}) \right) < \left( a(\mathscr{L}), b(\mathscr{L}) \right),
		\end{equation} 
		in the lexicographic ordering. It is called \emph{balanced} if this property holds for every such $H'\subset G$.
	\end{definition}
	In this article we confirm the following case of Manin's conjecture.
	
	\begin{thm}\label{thm:main_thm}
		Let $L$ be a simply-connected absolutely-simple connected algebraic group over a number field $F$, and $H$ be a semisimple absolutely maximal connected $F$-subgroup of $L$. Let $G=L^n$, and $\Delta(H)$ the image of the diagonal embedding of $H$ into $G$. Let $X$ be a $G$-equivariant compactification of $X^{\circ}=\Delta(H)\backslash G$. Let $\mathscr{L}$ be a balanced line bundle on $X$, with a smooth adelic metrization $\mathcal{L}$. Suppose that for any completion $F_v$ of $F$, the map of Galois cohomology $H^1(F_v,H)\to H^1(F_v,L)$ is injective. Then
		\begin{equation}\label{eq:asymptotic}
		N(X^{\circ}(F),\mathcal{L},T) \sim c_{\mathcal{L}}\cdot T^{a(\mathscr{L})}(\log T)^{b(\mathscr{L})-1},
		\end{equation}
		as $T\to\infty$, for some $c_{\mathcal{L}}>0$.
	\end{thm}
	Our proof is based heavily on \cite{GTT}, and also combines techniques from \cite{CT}\cite{GO}.
	
	\begin{example}
		Let $L=\SL_{2m}$ and $H=\Sp_{2m}$, where $m\geq2$ is an integer. Let $X$ be an equivariant compactification of $\Delta(H)\backslash L^n$, and we take $\mathscr{L}$ to be its anticanonical bundle $-K_X$. By \cite[Theorem 1.3]{HTT} and \Cref{intermediateclassification} below, $-K_X$ is balanced. Moreover, since $H^1(F_v,H)$ is trivial, we know that $H^1(F_v,H)\to H^1(F_v,L)$ is injective. Hence the conditions of \Cref{thm:main_thm} are satisfied, and the asymptotic formula \eqref{eq:asymptotic} holds for this case.
	\end{example}
	
	\subsection*{Acknowledgement}
	I would like to thank David Anderson and Alexander Gorodnik for helpful discussions.
	
	\section{Intermediate Subgroups}
	Let $F$ be an algebraically closed field of characteristic $0$. Let $L$ be a simply-connected simple connected algebraic group defined over $F$, and $H$ be a semisimple maximal connected $F$-subgroup of $L$. $G=L^n$ is the $n$-fold direct product of $L$. Let $\Delta(H)$ denote the diagonal embedding of $H$ into $G$. In this section we classify all the subgroups of $G$ which contain $\Delta(H)$. For any group $N$ and integer $r$, let $\Delta_{r}(N)$ denote the image of the diagonal embedding of $N$ into $N^r$.
	\begin{prop}[c.f.\cite{GTT} Proposition 4.1]\label{intermediateclassification}
		Suppose $M$ is a connected algebraic group such that $\Delta(H)\subset M\subset G$. Then there exists positive integers $n_1, \cdots , n_k; m_1, \cdots , m_l$ such that $n_1 + \cdots + n_k + m_1 + \cdots + m_l = n$, and that up to permutation of indices, $M$ is the image of the morphism
		$$\prod_{i=1}^{k}\Delta_{n_i}(H)\prod_{j=1}^{l}\Delta_{m_j}(L) \longrightarrow L^n$$
		$$(h_1, \cdots , h_s, g_1, \cdots, g_t) \mapsto (h_1, \cdots, h_s, \rho_1(g_1), \cdots, \rho_t(g_t)),$$
		where $\rho_i : L \rightarrow L$ is an automorphism of $L$ fixing each element in $H$.
	\end{prop}
	\begin{proof}
		We prove by induction on $n$. The case $n=1$ follows from the maximality of $H$. Suppose the proposition holds for $n$. For $G=L^n$, let $p_1: M \rightarrow L^{n-1}$ denote the projection onto the first $n-1$ entries, and $p_2: M\rightarrow L$ the projection onto the last entry. Without loss of generality we may assume $p_2(M)=L$. Otherwise $M$ is contained in $H^n$ and the conclusion holds by \cite{GTT} Proposition 4.1.
		
		\begin{center}
			\begin{tikzcd}[column sep=small]
			& M \arrow[dr, "p_2"] & \\
			L^{n-1} \arrow[ur,"p_1", leftarrow] & & L
			\end{tikzcd}
		\end{center}
		
		Now consider $N:= p_2(ker~p_1)$. Since $M$ contains $\Delta(H)$, it follows that $N$ is normalized by $H$. If $N$ is contained in $H$, then $N$ is a normal subgroup of $H$. If $N$ is not contained in $H$, then $H$ is a proper subgroup of $NH$, but $H$ is maximal, hence $NH = L$. Since $L$ is simple and $N$ is normal in $L$, we conclude that $N=L$. Now we discuss all the possible cases.
		
		\textit{Case 1.}~$N={L}$. In this case, $M=p_1(M)\times L$. By inductive hypothesis we know that $p_1(M)$ is of the form as in the proposition. Hence $M=p_1(M)\times L$ also satisfies the conclusion.
		
		\textit{Case 2.}~$N$ is an infinite normal subgroup of $H$. Let $N_L(N)$ denote the normalizer of $N$ in $L$, then it is a proper subgroup of $L$ containing $H$. Since $p_2(M) = L$, we can take $g\in p_2(M)$ such that $g\not\in N_{L}(N)$. Then there exists $a\in L^{n-1}$ such that $(a, g)\in M$. It follows that $(a^{-1}, Ng^{-1})\subset M$, and this implies $gNg^{-1}\subset N$. Therefore $g\in N_L(N)$, contradicting the choice of $g$.
		
		\textit{Case 3.}~$N$ is finite. Then $p_1$ is an isogeny. By inductive hypothesis we may assume that $p_1(M)=L^rH^s$. Since $p_2(M)=L$, we get a surjection $\overline{p_2} : L^rH^s\rightarrow L/N$, whose kernel is denoted by $K$. Since $L$ and $H$ are both semisimple, by Lemma \ref{prodnormal}, the neutral component $K^0$ of $K$ is of the form $L^{r-1}H^s$.\\
		
		\begin{center}
			\begin{tikzcd}
			& p_1(M)\simeq L^rH^s \arrow[d] \arrow[dr, "\tilde{p_2}"] & & \\
			& p_1(M)/K_0 \arrow[r,dotted, "\phi"] \arrow[d, "\pi"]
			& L \arrow[d, "\pi"] \\
			& p_1(M)/K \arrow[r, "\psi"]
			& L/N
			\end{tikzcd}
		\end{center}
		
		Since $\psi$ is an automorphism of $L/N$, by the following Theorem \ref{PR2.8}, $\psi$ induces an isomorphism $\phi:p_1(M)\rightarrow L$. Further lifting to $p_1(M)$, we can say that $\overline{p_2}$ induces $\tilde{p_2} : p_1(M) \rightarrow L$, whose mapping graph $M_0$ is contained in $M$. Since $M$ is connected, we know that $M_0 = M$. In other words, $p_1$ is surjective. Therefore $\overline{p_2}$ induces an automorphism $\rho: L\rightarrow L$. Since $\Delta(H)$ is contained in $M$, we know that $\rho$ fixes each element in $H$.  \par
	\end{proof}
	
	\begin{lem}\label{prodnormal}
		Suppose $G_i$'s are simple algebraic groups, and $H$ is a connected normal subgroup of $\prod_{i=1}^{n}G_i$. Then there exist $1\leq i_1 < \cdots < i_k \leq n$ such that $H = \prod_{j = 1}^{k}G_{i_j}$.
	\end{lem}
	
	\begin{proof}
		We prove by induction on $n$. Without loss of generality, we may assume that the projection $p_2(H)$ of $H$ onto the last summand is nontrivial. Since $G_n$ is simple, and $p_2(H)$ is connected normal in $G_n$, we know that $p_2(H) = G_n$. Suppose $h = (h_1, h_2)$ where $h_1 \in \prod_{i=1}^{n-1} G_i$. Then for any $g \in G_n$ we have $(h_1, gh_2 g^{-1}) \in H$, as $H$ is normal. This implies that $gh_2g^{-1}h_2^{-1}\in p_2(ker p_1)$. Hence $p_2(ker~p_1)$ contains $[G_n, G_n]$, the commutator group of $G_n$. Since $G_n$ is simple, we know that $p_2(ker~p_1) = G_n$. Therefore $H = p_1(H) \times G_n$. And by inductive hypothesis $p_1(H)$ is already a direct product.
	\end{proof}
	
	We now see that the study of automorphism group plays an import role here. Hence let us recall the following result.
	
	\begin{thm}(\cite{PR} Theorem 2.8)\label{PR2.8}
		For any simply connected semisimple group G, the automorphism group Aut $G$ is the semidirect product of Int $G \simeq \bar{G}$ by Sym$(R)$, where $R$ is the Dynkin diagram.
		If $G$ is an arbitrary semisimple group and $\tilde{G} \overset{\pi}{\rightarrow} G$ is a universal covering, then Aut $G$ is isomorphic to the subgroup of Aut $\tilde{G}$ fixing ker $\pi$ , the fundamental group.
	\end{thm}
	
	Here we also give a corollary of this theorem, which will be crutial in the next section.
	
	\begin{cor}\label{finitelymanyintermediate}
		Let $G$ be a simply-connected semisimple group, and $H$ is a maximal connected subgroup of $G$. Then the set
		$$\{ \rho \in Aut \, G : \rho(h)=h,\forall h \in H \}$$
		is finite.
	\end{cor}
	
	\begin{proof}
		By Theorem \ref{PR2.8}, it suffices to consider inner automorphisms only, as $\mathrm{Sym}\,R$ is already finite. Suppose $\rho (g) = g_0gg_0^{-1}$ for some $g_0 \in G$, then $g_0 \in Z_H(G)$. Since $H$ is maximal and $G$ is semisimple, we know that $Z_H(G)$ is finite. Hence we only have finitely many choices of such automorphisms.
	\end{proof}
	
	\section{Equidistribution of translated measures}\label{sect:equidistribution}
	Let $F$ be a number field. Let $L$ be a connected semisimple group defined over $F$, and $H$ a semisimple maximal connected subgroup of $L$. Notice that in this section we drop the assumption that $L$ is simply-connected. Let $\pi:\tilde{L}\rightarrow L$ be the universal cover of $L$ and $W$ a compact subgroup of $L(\mathbb{A})$ such that $W\cap L(\mathbb{A}_f)$ is open in $L(\mathbb{A}_f)$. Define $$L_W:=L(F)\pi(\tilde{L}(\mathbb{A}))W$$ and $$Y_W:=L(F)\backslash L_W.$$ \\
	Let $C_c(Y_W)^W$ denote the space of compactly supported $W$-invariant continuous functions on $Y_W$. Let $G=L^n$ be the $n-$fold direct product of $L$, and $V=W^n$ is an open subgroup of $G$. Set
	\begin{equation*}
	G_V:=G(F)\pi(\tilde{G}(\mathbb{A}))V=L_W\times\cdots L_W.
	\end{equation*}
	\begin{thm}\label{equidistribution}
		Let $L$ be a connected simple group defined over $F$. Let $H$ be a simple maximal connected subgroup of $L$. Suppose $$\{(b_1^{(k)}, \cdots, b_n^{(k)})\}\subset L_W \times \cdots \times L_W$$ is a sequence such that
		\begin{enumerate}[(1)]
			\item For any $i\neq j$, $$\lim_{k\rightarrow \infty}(b_i^{(k)})^{-1}b_j^{(k)}=\infty.$$
			\item For any $i$, $b_i^{(k)}\rightarrow\infty$ modulo $H(\mathbb{A})$, as $k\rightarrow\infty$.
		\end{enumerate}
		Then for all $f_1, \cdots , f_n \in C_c(Y_W)^W$, we have
		$$\lim_{k\rightarrow \infty}\int_{Y_W}f_1(yb_1^{(k)})\cdots f_n(yb_n^{(k)})\mathrm{d}\nu(y)=\int_{Y_W}f_1\mathrm{d}\mu\cdots\int_{Y_W}f_n\mathrm{d}\mu$$
		where $\nu$ is the invariant probability measure supported on $H(F)\backslash(H(\mathbb{A})\cap L_W)$ considered as a measure on $Y_W$ via pushing forward by the natural injection, and $\mu$ is the probability Haar measure on $Y_W$.
	\end{thm}
	
	\begin{proof}
		Set 
		$$W^{(k)}=\bigcap_{i=1}^{n}b_i^{(k)}W(b_i^{(k)})^{-1}.$$
		From the proof of Corollary 4.14 in \cite{GO}, we know that $H(F)\pi(\tilde{H}(\mathbb{A}))(W^{(k)}\cap H(\mathbb{A}_f))$ is a normal subgroup of $H(\mathbb{A})\cap L_W$ with finite index, for any $k$. Hence there exists a finite subset $\Delta^{(k)}\subset H(\mathbb{A})\cap L_W$ such that
		$$H(\mathbb{A})\cap L_W=\bigcup_{x\in \Delta^{(k)}}H(F)\pi(\tilde{H}(\mathbb{A}))x(W^{(k)}\cap H(\mathbb{A}_f)),$$
		where the union is a disjoint union. We note that by Corollary 4.10 in \cite{GO}, $H(F)\pi(\tilde{H}(\mathbb{A}))$ is normal in $H(\mathbb{A})\cap L_W$. Observe that the function 
		\[
		y\mapsto f_1(yb_1^{(k)})\cdots f_n(yb_n^{(k)})
		\]
		is right invariant under $W^{(k)}$.Therefore
		
		\begin{equation}\label{flattenedintegral}
		\begin{split}
		& \int_{Y_W}f_1(yb_1^{(k)})\cdots f_n(yb_n^{(k)})\mathrm{d}\nu(y) \\
		=& \sum_{x\in\Delta^{(k)}}\int_{W^{(k)}}\int_{x_0\pi(\tilde{H}(\mathbb{A}))x}f_1(uwb_1^{(k)})\cdots f_n(uwb_n^{(k)})\mathrm{d}\mu_{x}^{(k)}(u)\mathrm{d}w \\
		=& \sum_{x\in \Delta^{(k)}}\frac{1}{\#\Delta^{(k)}}\int_{x_0\pi(\tilde{H}(\mathbb{A}))x}f_1(ub_1^{(k)})\cdots f_n(ub_n^{(k)})\mathrm{d}\mu_{x}^{(k)}(u),
		\end{split}
		\end{equation}
		where $\mu_{x}^{(k)}$ is the invariant probability measure supported on $x_0\pi(\tilde{H}(\mathbb{A}))x$, and $\mathrm{d}w$ is the normalized invariant measure on $W^{(k)}$.
		
		Now for any $k$, choose $x^{(k)}\in \Delta^{(k)}$, and set
		\begin{equation}
		c^{(k)}=(x^{(k)}b_1^{(k)}, \cdots, x^{(k)}b_n^{(k)}).
		\end{equation}
		Since $(c_i^{(k)})^{-1}c_j^{(k)}=(b_i^{(k)})^{-1}b_j^{(k)}$ and $x^{(k)}\in H(\mathbb{A})$, we can see that the sequence $\{c^{(k)}\}$ still satisfies both conditions in the theorem. According to the definition we can rewrite the last integral in (\ref{flattenedintegral}) as
		\begin{equation}\label{tensor}
		\int_{x_0\pi(\tilde{H}(\mathbb{A}))x}f_1(ub_1^{(k)})\cdots f_n(ub_n^{(k)})\mathrm{d}\mu_{x}^{(k)}(u)=\int_{G(F)\backslash G_V}f_1\otimes\cdots\otimes f_n\mathrm{d}(c^{(k)}\cdot\lambda_{H})
		\end{equation}
		Now it remains to determine the limit points of $c^{(k)}\cdot \lambda_{H}$, where $\lambda_{H}$ is the invariant probability measure on $\pi(\Delta\tilde{(H)}(\mathbb{A}))$.
		
		Since $H$ is maximal in $L$, the centralizer $Z_L(H)$ of $H$ in $L$ is anisotropic. Hence the centralizer of $\Delta(H)$ in $G=L^n$, which equals $Z_L(H)\times\cdots Z_L(H)$, is also anisotropic over $F$. Therefore, by \cite{GO} Theorem 1.7(1) we know that $\{c^{(k)}\cdot \lambda_{H}\}$ is relatively compact. Suppose $\mu$ is a limit point. By \cite{GO} Theorem 1.7(2) we get a connected $F$-subgroup $M$ of $G$ and a sequence $\delta^{(k)}\in G(F)$ such that $$\Delta(H)\subset (\delta^{(k)})^{-1}M\delta^{(k)}\subset G.$$ Hence
		\begin{equation}\label{conjugationM}
		M=\delta^{(k)}N_k(\delta^{(k)})^{-1},
		\end{equation}
		where $N_k$ is an intermediate subgroup as described in Proposition \ref{intermediateclassification}. Now it suffices to show that $M=G$, and the theorem will follow by a similar argument to \cite[Theorem 5.1]{GTT}. 
		
		We prove $M=G$ by contradiction. Suppose $M$ is a proper subgroup of $G$. By Corollary \ref{finitelymanyintermediate}, the number of intermediate subgroups is finite, and thus by passing to a subsequence we may assume that $N_k=N$ for all $k$. Here $N$ is a proper subgroup of $G=L^n$.
		
		\textit{Case 1. There is no $\Delta_{n_i}(H)$ part.} Since $N$ is a proper subgroup of $G$, we can find $i\neq j$ such that $\pi_{ij}(N)=\Delta_2(L)$. Set $\sigma^{(k)}=(\delta^{(1)})^{-1}\delta^{(k)}$, we see from Equation (\ref{conjugationM}) that $$z^{(k)}:=(\sigma_i^{(k)})^{-1}\sigma_j^{(k)}\in Z(L)(F).$$By \cite{GO} Theorem 1.7(2), there exists $h^{(n)} \in \pi(\Delta(\tilde{H})(\mathbb{A}))$ such that $\delta^{(k)}h^{(k)}c^{(k)}$ converges. In particular, $(z^{(k)})^{-1}(c_i^{(k)})^{-1}c_j^{(k)}$ converges. Since $Z(L)$ is finite, $\{z^{(k)}\}$ is a compact set. Hence $(c_i^{(k)})^{-1}c_j^{(k)}$ converges, but this contradicts to the fact that pairwise ratios diverge.
		
		\textit{Case 2. There is a $\Delta_{n_i}(H)$ part.} We may assume that $\pi_1(N)=H$. We see from \eqref{conjugationM} that $\sigma_1^{(k)}$ is in the normalizer $N_L(H)$ of $H$. But $H$ has finite index in $N_L(H)$, hence $\{\sigma_1^{(k)}\}$ is bounded modulo $H$. Again by \cite{GO} Theorem 1.7, the sequence $\sigma_1^{(k)}h_1^{(k)}c_1^{(k)}$ converges. This contradicts to the facts that $\{c_1^{(k)}\}$ diverges modulo $H$.
		
		Therefore, $M=G$, and still by \cite{GO} Theorem 1.7 we know that there exists a normal subgroup $M_0$ of $M(\mathbb{A}) = G(\mathbb{A})$ containing $G(F)\pi(\tilde{G(\mathbb{A})})$ and $g\in\pi(\tilde{G(\mathbb{A})})$ such that for any $f\in C_c(G(F)\backslash G_V)^V$, we have the following
		\begin{equation}\label{limitMeasureEqualsHaar}
		\begin{split}
		&\int_{G(F)\backslash G_V}f\,\mathrm{d}\mu = \int_{G(F)\backslash G_V}f\,\mathrm{d}(g\cdot \nu_{M_0}) = \int_{G(F)\backslash G_V}f\,\mathrm{d}\nu_{M_0}\\
		=&\int_{G(F)\backslash G_V}\int_{V}f(uv)\,\mathrm{d}v\,\mathrm{d}\nu_{M_0}=\int_{G(F)\backslash G_V}f\,\mathrm{d}z,
		\end{split}
		\end{equation}
		where $\mathrm{d}\mu_{M_0}$ is the pushforward of the Haar measure on $x_0M_0$, and $\mathrm{d}z$ is the Haar measure on $G(F)\backslash G_V$.\\
		
		Finally, combining equations (\ref{flattenedintegral})(\ref{tensor})(\ref{limitMeasureEqualsHaar}) we finish the proof of the theorem.
	\end{proof}
	
	\section{Volume Computation}
	In this section, let $L$ be a simply-connected simple connected algebraic group over a number field $F$, and $H$ be a simple maximal connected $F$-subgroup of $L$. Denote by $G$ the $n$-fold direct product of $L$. We treat $H$ as a subgroup of $G$ via diagonal embedding. Let $X$ be a smooth projective equivariant compactification of $X^{\circ}=H\backslash G$. Let $\mathscr{L}$ be an ample line bundle on $X$. By \cite{GTT} Proposition 2.1, we can write
	\begin{equation}
	\mathscr{L}=\sum_{\alpha\in\mathcal{A}}\lambda_{\alpha}D_{\alpha},\quad \lambda_{\alpha}\in\Q_{>0},
	\end{equation}
	and
	\begin{equation}
	-K_X=\sum_{\alpha\in\mathcal{A}}\kappa_{\alpha}D_{\alpha},
	\end{equation}
	where all $\kappa_{\alpha}\geq 1$. \par 
	Given a smooth metrization $\mathcal{L}$ of $\mathscr{L}$, as in \cite{CT} Section 2.1 we have a corresponding height function
	\begin{equation}
	\mathrm{H}=\mathrm{H}_{\mathcal{L}}\colon X^{\circ}(F)\to\R_{>0}.
	\end{equation} \par
	There exists a compact open subgroup $V$ of $G(\A_f)$ such that the adelic height function $\mathrm{H}$ is invariant under $V$. By possibly replacing $V$ with a smaller compact open subgroup, we may assume that $V=W\times\cdots\times W$ for a compact open subgroup $W$ of $H(\A_f)$.\\

	To compute the volume of the height ball via standard Tauberian argument, we need the following result.
	\begin{thm}[\cite{GTT} Theorem 3.3]\label{thm:height_integral}
		Let $G$ be a connected semisimple algebraic group and $H\subset G$ a closed subgroup, defined over a number field $F$, such that the map $H^1(E, H)\rightarrow H^1(E, G)$ is injective, for $E$ being either $F$ or a completion of $F$. Let $X$ be a smooth projective equivariant compactification of $X^{\circ}=H\backslash G$ with normal crossing boundary $\bigcup_{\alpha\in\mathcal{A}}D_{\alpha}$ and
		$$\mathrm{H}:\mathbb{C}^{\mathcal{A}}\times X^{\circ}(\mathbb{A})\rightarrow \mathbb{C}$$
		an adelic height system. Then there exists a function $\Phi$, holomorphic and bounded in vertical strips for $\text{Re }s_{\alpha}> \kappa_{\alpha}-\epsilon$, for some $\epsilon > 0$, such that for $\mathbf{s}=(s_{\alpha})$ in this domain one has
		$$\int_{X^{\circ}(\mathbb{A})}\mathrm{H}(\mathbf{s}, x)^{-1}\mathrm{d}x=\prod_{\alpha\in\mathcal{A}}\zeta_{F}(s_{\alpha}-\kappa_{\alpha}+1)\cdot\Phi(\mathbf{s}),$$
		where $\zeta_{F}$ is the Dedekind zeta function.
	\end{thm}
	
	Let $B_T$ be the height ball in $X_V=X^{\circ}(\mathbb{A})$ defined by
	\begin{equation}
	B_T=B_{T,\mathcal{L}}=\left\{ x\in X_V\colon \mathrm{H}_{\mathcal{L}}(x)<T \right\}.
	\end{equation}
	
	We have the following asymptotic formula for the volume of the height ball.
	
	\begin{lem}[\cite{GTT} Lemma 6.3]\label{lem:volume_asymptotics}
		Let $\mathscr{L}$ be an ample line bundle on $X$, and $\mathcal{L}$ be a smooth metrization of $\mathscr{L}$. Then
		\begin{equation}
		\mathrm{vol}(B_T)\sim c_{\mathcal{L}}\cdot T^{a_{\mathscr{L}}}(\log T)^{b_{\mathscr{L}}-1}
		\end{equation}
		with $a_{\mathscr{L}}$,$b_{\mathscr{L}}$ as in \eqref{eq:def_a_b} and $c_{\mathcal{L}}>0$.
	\end{lem}
	
	The following lemma is a generalized version of \cite{GTT} Lemma 6.6.
	\begin{lem}\label{lem:volume_ratio_tends_to_0}
		Let $H\subset M \subset G$ be semisimple connected algebric groups defined over $F$, and $\mathscr{L}$ be a balanced line bundle on $H\backslash G$. Let $K$ be a compact subset of $G(\mathbb{A})$ such that $k_1k_2^{-1}\notin M(\mathbb{A})$ for all distinct $k_1,k_2\in K$. Suppose $H^1(F_v, H) \rightarrow H^1(F_v, G)$ is injective for any place $v$ of $F$, then for any smooth adelic metrization of $\mathscr{L}$, we have
		\begin{equation}
		\lim_{T\rightarrow\infty}\frac{\mathrm{vol}\left(B_T\cap (H\backslash M)(\mathbb{A})\cdot K\right)}{\mathrm{vol}(B_T)}=0
		\end{equation}
	\end{lem}
	
	\begin{proof}
		Since we have injectivity of cohomology, we can apply \Cref{thm:height_integral} to $H\backslash M$ and $H\backslash G$. Now we know the poles of the height integral and their orders, and we can apply the standard Tauberian argument to obtain the volume asymptotics (see \cite{CT} Appendix A). The proof of \cite{GTT} Lemma 6.6 works with no changes needed. 
	\end{proof}
	
	\begin{cor} Let $G=L^n$ and $H$ embeds in $G$ diagonally. Let the volume be given by the Tamagawa measure with respect to a balanced line bundle $\mathscr{L}$.
		\begin{enumerate}[(1)]
			\item Let $K_1$ be a compact subset of $L(\mathbb{A})$. For any $1\leq i < j \leq n$, we have
			$$\lim_{T\rightarrow\infty}\frac{\mathrm{vol}(B_T\cap \{(x_1, \cdots , x_n)\in (H\backslash G)(\mathbb{A}) : (x_i)^{-1}x_j\in K_1\})}{\mathrm{vol}(B_T)}=0$$
			\item Let $K_2$ be a compact subset of $H\backslash L(\mathbb{A})$. Fix $1\leq i \leq n$, we have
			$$\lim_{T\rightarrow\infty}\frac{\mathrm{vol}(B_T\cap \{(x_1, \cdots , x_n)\in (H\backslash G)(\mathbb{A}) : x_i\in K_2\})}{\mathrm{vol}(B_T)}=0$$
		\end{enumerate}
	\end{cor}
	
	\begin{proof}
		In (1) we take
		\begin{equation}
		M=\{ (x_1,\cdots,x_n)\in G\colon x_i = x_j \},
		\end{equation}
		and
		\begin{equation}
		K=\{ (x_k)\in G(\mathbb{A})\colon x_j\in K_1;\,x_k=e,\forall k\neq j \}.
		\end{equation}\par 
		In (2) we take
		\begin{equation}
		M=\{ (x_1,\cdots,x_n)\in G\colon x_i \in H \},
		\end{equation}
		and
		\begin{equation}
		K=\{ (x_k)\in G(\mathbb{A})\colon x_i\in K_2;\,x_k=e,\forall k\neq i \}.
		\end{equation}\par 
		Now it remains to apply \Cref{lem:volume_ratio_tends_to_0}.
	\end{proof}
	
	We recall the definition of $L_W$, $Y_W$ and $G_V$ from \Cref{sect:equidistribution}. Since $L$ is simply-connected, we have $L_W=L(F)\pi(\tilde{L}(\mathbb{A}))W=L(\mathbb{A})$, $G_V=(L_W)^n=G(\mathbb{A})$. Denote $Y = Y_W = L(F)\backslash L(\mathbb{A})$, and $Z = Z_V = G(F)\backslash G(\mathbb{A})$. Let $\nu$ be the invariant probability measure supported on $H(F)\backslash H(\mathbb{A})$. Let $\mathrm{d}x$ denote the Tamagawa measures on $X^{\circ}(\mathbb{A})$, and $\mathrm{d}z$ denote the invariant probability measure on $G(F)\backslash G(\mathbb{A})$. \par 
	We recall the following result from \cite{GTT}.
	
	\begin{prop}[\cite{GTT} Corollary 6.8]\label{prop:average_Haar}
		For any $f\in C_c(Z)$,
		$$\lim_{T\rightarrow \infty}\frac{1}{\mathrm{vol}(B_T)}\int_{B_T}\mathrm{d}x\int_{Y}f(yx)\mathrm{d}\nu(y)=\int_{Z}f \mathrm{d}z.$$
	\end{prop}
	
	\begin{proof}
		We follow the proof of \cite{GTT} Corollary 6.8. By the Stone-Weierstrass theorem, it suffices to consider functions of the form $f=f_1\otimes\cdots\otimes f_n$ with $f_i\in C_c(Y)$, and it's shown that we may assume $f_i$ to be $W$-invariant. In this case,
		\begin{equation}
		I(x)=\int_{Y}f(yx)\mathrm{d}\nu(y)=\int_{Y}f_1(yx_1)\cdots f_n(yx_n)\mathrm{d}\nu(y).
		\end{equation} \par 
		Given compact subsets $K_1$ of $L(\mathbb{A})$ and $K_2$ of $(H\backslash L)(\mathbb{A})$, we set
		\begin{equation}
		B_{T,K_1,K_2}=\{ x\in B_T\colon (x_i)^{-1}x_j\notin K_1,i\neq j;\;x_i\notin K_2 \}.
		\end{equation}
		By \Cref{equidistribution}, for every $\epsilon>0$, there exists $K_1$ and $K_2$ such that for all $x=(x_1,\cdots,x_n)\in B_{T,K_1,K_2}$, we have
		\begin{equation}
		\left\vert I(x)-\int_{Y}f_1\mathrm{d}\mu \cdots\int_{Y}f_n\mathrm{d}\mu  \right\vert < \epsilon,
		\end{equation}
		and
		\begin{equation}\label{eq:temp1}
		\int_{B_{T,K_1,K_2}}I(x)\mathrm{d}x = \mathrm{vol}(B_{T,K_1,K_2})\int_{Y}f_1\mathrm{d}\mu \cdots\int_{Y}f_n\mathrm{d}\mu +O(\epsilon\mathrm{vol}(B_{T,K_1,K_2})).
		\end{equation}
		One also has
		\begin{equation}\label{eq:temp2}
		\int_{B_T\backslash B_{T,K_1,K_2}}I(x)\mathrm{d}x = O(\mathrm{vol}(B_T\backslash B_{T,K_1,K_2})).
		\end{equation}
		Since the line bundle $\mathscr{L}$ is balanced, it follows from \Cref{lem:volume_ratio_tends_to_0} that
		\begin{equation}
		\frac{\mathrm{vol}(B_T\backslash B_{T,K_1,K_2})}{\mathrm{vol}(B_T)}\to 0,\quad T\to\infty.
		\end{equation}
		Since $\epsilon>0$ is arbitrary, combining \eqref{eq:temp1} and \eqref{eq:temp2} we finish the proof of the proposition.
	\end{proof}
	
	\begin{proof}[Proof of \Cref{thm:main_thm}]
		By \cite{GTT} Lemma 3.4, the height balls are well-rounded. Then \Cref{thm:main_thm} follows from \Cref{prop:average_Haar} via the standard unfolding argument. See e.g. \cite{GO} Proposition 5.3 and \cite{GTT} Theorem 6.9.
	\end{proof}

	\subsection*{Availability of data and materials}
	Data sharing not applicable to this article as no datasets were generated
	or analysed during the current study.

	%\section{Geometry and Cohomology}
	%\begin{lem}
	%Suppose $L \rightarrow H\backslash L$ admits a rational section. Then $L^n \rightarrow \Delta(H)\backslash L^n$ also admits a rational section.
	%\end{lem}
	
	%\begin{proof}
	%Suppose $\sigma$ is a rational section of $L \rightarrow H\backslash L$. Then there exists a Zariski open subset $U$ of $H\backslash L$ such that
	%$$\sigma: U \rightarrow L$$
	%is a regular section. \par
	%Consider the projection
	%$$p_1: \Delta(H)\backslash L^n \rightarrow H \backslash L$$
	%$$\overline{(l_1, \cdots, l_n)}\mapsto \overline{l_1}.$$
	%It's a regular map. Hence $p_1^{-1}(U)$ is a Zariski open subset of $\Delta(H)\backslash L^n$. Now we define
	%$$\tilde{\sigma}: p_1^{-1}(U) \longrightarrow L^n$$
	%as $\tilde{\sigma}(\overline{(l_1, \cdots, l_n)}) = {(\sigma(l_1), \sigma(l_1)l_1^{-1}l_2, \cdots, \sigma(l_1)l_1^{-1}l_n)}$. One can verify directly that this map is well-defined, and is indeed a rational section as required.
	%\end{proof}
	
	\bibliographystyle{alpha}
	\bibliography{rational}

\end{document}